\documentclass[12pt]{amsart}
\usepackage{amssymb,amsmath,amsthm,amsfonts,amsopn,url,color,hyperref,enumerate,mathtools}
\usepackage[normalem]{ulem}
\usepackage{amscd}
\usepackage[mathscr]{euscript}
\usepackage{mathtools,rotating}
\usepackage[all,2cell]{xy}
\UseAllTwocells{}
\usepackage{pgf,tikz}
\usepackage{mathrsfs}
\usetikzlibrary{arrows}
\usetikzlibrary{patterns}
\usepackage{tikz,pgfplots}
\usetikzlibrary{matrix, cd}
\usepackage{esvect}
\usepackage{tabularx}
\usepackage{arydshln}
\usepackage{array}
\newtheorem{theorem}{Theorem}[section]

\newtheorem{proposition}[theorem]{Proposition}

\theoremstyle{definition}\newtheorem{remark}[theorem]{Remark}
\newtheorem*{question*}{Question}
\theoremstyle{definition}\newtheorem{definition}[theorem]{Definition}

\newcommand{\mono}{\ar@{^{(}->}}

\author{Mike Hensler}
\address{Department of Mathematical Sciences \\ Indiana University East \\ Richmond, IN 47374}
\email{michens@iu.edu}

\author{Hannah Klawa}
\address{Department of Mathematical Sciences \\ Indiana University East \\ Richmond, IN 47374}
\email{hklawa@iu.edu}

\title{Graded Injective Domains}

\begin{document}

\begin{abstract}
  An integral domain \(R\) is an $i$-domain if for every overring \(S\) of \(R\),
  \(\text{Spec}(S) \rightarrow \text{Spec}(R)\) is injective and is a mated integral if for every overring \(S\) of \(R\)
  and prime ideal \(P\) of \(R\) such that \(PS \neq S\), there exists exactly one prime ideal \(Q\) of \(S\) such that
  \(Q \cap R = P\).
  In this paper, we explore graded notions of $i$-domains and mated domains and their connection with gr-Pr\"{u}fer domains.

\end{abstract}

\maketitle
\section{Introduction}

Let $R$ be an integral domain with fraction field $K$. An \emph{overring} of $R$ is a integral domain $S$ such that $R \subseteq S \subseteq K$. Recall from~\cite{DobDaw-GDpoly} that an extension of integral domains $R \subseteq S$ is a \emph{mated extension} if for every  prime ideal \(P\) of \(R\) such that \(PS \neq S\), there exists exactly one prime ideal \(Q\) of \(S\) such that \(Q \cap R = P\) and an integral domain $R$ is said to be a \emph{mated domain} if $R \subseteq S$ is a mated extension for every overring $S$ of $R$. In~\cite[Proposition 3.6]{Dob-going-down-simple2}, Dobbs proves that for an integrally closed domain $R$, $R$ is Pr\"{u}fer if and only if it is a mated integral domain.

In~\cite{Pap-topologically}, Papick introduces injective extensions and studies results for injective extensions and domains. An extension $R \subseteq S$ of integral domains is an \emph{injective extension }(or \emph{$i$-extension}) if the canonical map $\text{Spec}(S) \rightarrow \text{Spec}(R)$ is injective and an integral domain $R$ is an \emph{$i$-domain} if $R \subseteq S$ is an injective extension for every overring $S$ of $R$. When an extension satisfies going-down, the notions of $i$-extensions and mated extensions coincide~\cite[Lemma 2.2]{Pap-topologically}.

A graded analog of injective extensions was introduced in~\cite{Mer-propgradover}. The purpose of this paper is to further study graded injective extensions as well as a notion of graded mated extensions and give give some connections to gr-Pr\"{u}fer domains.

\section{Preliminary Material}
A \emph{torsionless grading monoid} is a torsion-free commutative cancellative monoid. Throughout this paper, we will let \( \Gamma \) be a torsionless grading monoid. An integral domain \(R\) is a \emph{\(\Gamma \)-graded domain} if \[R = \bigoplus_{\alpha \in \Gamma}R_{\alpha}\]
where each \(R_{\alpha}\) is an additive subgroup of \(R\) and \(R_{\beta}R_{\gamma} \subseteq R_{\beta + \gamma}\) for
all \(\beta, \gamma \in \Gamma \). A nonzero element \(r \in R_{\alpha}\) is \emph{homogeneous} of degree \(\alpha \)
and each \(R_{\alpha}\) is a \emph{graded component} of \(R\). All graded integral domains in this paper will be assumed to be $\Gamma$-graded integral domains where $\Gamma$ is a torsionless grading monoid.

Let \[R = \bigoplus_{\alpha \in \Gamma}R_{\alpha},\] be a \(\Gamma \)-graded integral domain.
Then \(R_H\) where \[H = \{x \in R\, | \, x \neq 0\text{ is homogeneous}\} \] is the \emph{graded quotient field} of
\(R\) and is a \(\langle \Gamma \rangle = \{a - b\, | \, a, b \in \Gamma \} \)-graded domain where
\[(R_H)_{\alpha} = \left\lbrace\tfrac{f}{g}\,|\,f,g \textnormal{ are homogeneous},
  g \neq 0\textnormal{ and }\deg(f) - \deg(g) = \alpha\right\rbrace \]
for \(\alpha \in \langle \Gamma \rangle \).
The graded quotient field of an integral domain need not be a field. For example, the integral domain $R =\mathbb{Z}[x]$ is a $\mathbb{Z}$-graded domain with $R_\alpha = \mathbb{Z}x^\alpha$ for $\alpha \geq 0$ and $R_\alpha = \{0\}$ for $\alpha < 0$. The graded quotient field of $\mathbb{Z}[x]$ is $\mathbb{Q}[x, 1/x]$ which is not a field.

The graded (or homogeneous) primes of \(R\) will be denoted \(h\textnormal{-Spec}(R)\) and the graded maximal ideals of
\(R\) will be denoted \(h\textnormal{-Max}(R)\).
A graded integral domain \(R\) with a single graded maximal ideal is gr-local.

We will denote the integral closure of an integral domain \(R\) with \(R^\prime \).

In~\cite{Mer-propgradover}, graded $i$-extensions and graded $i$-domains are introduced. A ring extension \( R \subseteq T \) is a \textit{graded i-extension} if for
every pair of homogeneous prime ideals \( Q_1 \) and \( Q_2 \) of \( T \) such
that \( Q_1 \cap R = Q_2 \cap R \), we have that \( Q_1 = Q_2 \) (that is, $h\text{-Spec}(T) \rightarrow h\text{-Spec}(R)$ is injective).

\section{Graded Mated Domains}

It turns out when an extension satisfies going-down, it follows directly that the notions of $i$-extensions and mated extensions coincide as is stated without proof in~\cite[Lemma 2.2]{Pap-topologically}. We provide a proof below so the reader can see how the proof for the graded version is analogous with appropriate graded substitutions.

\begin{proposition}[\cite{Pap-topologically}, Lemma 2.2]\label{gd-mated} Assume that $R \subseteq T$ satisfies going-down. Then $R \subseteq T$ is an $i$-extension if and only if $R \subseteq T$ is mated.
\end{proposition}

\begin{proof}
  First, suppose \( R \subseteq T \) is mated.
  Let \( Q_1 \) and \( Q_2 \) be two distinct prime ideals of \( T \) such
  that \( Q_1 \cap R = Q_2 \cap R = P \) for some prime ideal \( P \) of
  \( R \).
  For contradiction, suppose \( PT = T \).
  Then, \( T = PT = (Q_1 \cap R)T \subseteq Q_1 \).
  This implies that \( Q_1 = T \), which is a contradiction.
  Therefore, \( PT \ne T \).
  Since \( R \subseteq T \) is mated, we have that \( Q_1 = Q_2 \).
  Therefore, \( R \subseteq T \) is an i-extension.

  Next, suppose \( R \subseteq T \) is an i-extension.

  Let \( P_1 \) be a prime ideal of \( R \) such that \( P_1T \ne T \).
  Then, \( P_1T \) is a proper ideal of \( T \).
  We have that \( P_1T \) must be contained in some maximal ideal \( Q_2 \) of
  \( T \).
  As \( Q_2 \) is maximal, it is also prime.
  Further, we have that \( P_2 = Q_2 \cap R \) is a prime ideal of \( R \),
  and \( P_1 \subseteq P_2 \).
  As \( R \subseteq T \) satisfies GD, we have some prime ideal \( Q_1 \) of
  \( T \) such that \( Q_1 \subseteq Q_2 \) and \( Q_1 \cap R = P_1 \).

  Now, let \( Q_1 \) and \( Q_2 \) be two distinct prime ideals of \( T \)
  such that \( Q_1 \cap R = Q_2 \cap R = P \) for some prime ideal \( P \) of
  \( R \).
  Since \( R \subseteq T \) is an i-extension, we have that \( Q_1 = Q_2 \).

  Therefore, \( R \subseteq T \) is mated.
\end{proof}

\begin{remark}
  This proof only uses the fact that \( R \subseteq T \) satisfies GD
  when showing that \( R \subseteq T \) is an i-extension implies that
  \( R \subseteq T \) is mated.
  It is not necessary to assume that \( R \subseteq T \) satisfies GD to show
  that \( R \subseteq T \) is mated implies that \( R \subseteq T \) is an
  i-extension; this holds for any ring extension.
\end{remark}

\begin{definition} A ring extension \( R \subseteq T \) is \textit{graded-mated} if for every
  homogeneous prime ideal \( P \) of \( R \) such that \( PT \ne T \), there
  exists a homogeneous prime ideal \( Q \) of \( T \) such that \( Q \cap R = P \).
  Further, for every pair of homogeneous prime ideals \( Q_1 \) and \( Q_2 \) of
  \( T \) such that \( Q_1 \cap R = Q_2 \cap R = P \), we have that \( Q_1 = Q_2 \).
\end{definition}

The notion of graded going-down was introduced by Sahandi and Shirmohommadi in~\cite[Definition 2.1]{Sah-gradedgd}. 

\begin{definition} We say that a ring extension \( R \subseteq T \) satisfies \textit{graded-GD} if
  for every pair of homogeneous prime ideals \( P_1 \) and \( P_2 \) of \( R \)
  such that \( P_1 \subseteq P_2 \), if there exists a homogeneous prime ideal
  \( Q_2 \) of \( T \) such that \( Q_2 \cap R = P_2 \), then there exists a
  homogeneous prime ideal \( Q_1 \) of \( T \) such that \( Q_1 \cap R = P_1 \)
  and \( Q_1 \subseteq Q_2 \).
\end{definition}

The following is a graded analog of~\cite[Lemma 2.2]{Pap-topologically}.

\begin{proposition}\label{gr-pap76-lemma2.2}
  Assume that \( R \subseteq T \) satisfies graded-GD\@.
  Then \( R \subseteq T \) is a graded i-extension if and only if
  \( R \subseteq T \) is graded-mated.
\end{proposition}

\begin{proof} The proof is parallel to that of Proposition~\ref{gd-mated} with appropriate graded substitutions.
  First, suppose \( R \subseteq T \) is graded-mated.
  Let \( Q_1 \) and \( Q_2 \) be two distinct homogeneous prime ideals of
  \( T \) such that \( Q_1 \cap R = Q_2 \cap R = P \) for some prime ideal
  \( P \) of \( R \).
  For contradiction, suppose \( PT = T \).
  Then, \( T = PT = (Q_1 \cap R)T \subseteq Q_1 \).
  This implies that \( Q_1 = T \), which is a contradiction.
  Therefore, \( PT \ne T \).
  Since \( R \subseteq T \) is graded-mated, we have that \( Q_1 = Q_2 \).
  Therefore, \( R \subseteq T \) is a graded i-extension.

  Next, suppose \( R \subseteq T \) is a graded i-extension.

  Let \( P_1 \) be a homogeneous prime ideal of \( R \) such that
  \( P_1T \ne T \).
  Then, \( P_1T \) is a proper homogeneous ideal of \( T \).
  We have that \( P_1T \) must be contained in some homogeneous maximal ideal
  \( Q_2 \) of \( T \).
  As \( Q_2 \) is homogeneous maximal, it is also homogeneous prime.
  Further, we have that \( P_2 = Q_2 \cap R \) is a homogeneous prime ideal of
  \( R \), and \( P_1 \subseteq P_2 \).
  As \( R \subseteq T \) satisfies graded-GD, we have some homogeneous prime
  ideal \( Q_1 \) of \( T \) such that \( Q_1 \subseteq Q_2 \) and
  \( Q_1 \cap R = P_1 \).

  Now, let \( Q_1 \) and \( Q_2 \) be two distinct homogeneous prime ideals of
  \( T \) such that \( Q_1 \cap R = Q_2 \cap R = P \) for some homogeneous prime
  ideal \( P \) of \( R \).
  Since \( R \subseteq T \) is a graded i-extension, we have that
  \( Q_1 = Q_2 \).

  Therefore, \( R \subseteq T \) is graded-mated.
\end{proof}

\begin{remark}
  This proof only uses the fact that \( R \subseteq T \) satisfies graded-GD
  when showing that \( R \subseteq T \) is a graded i-extension implies that
  \( R \subseteq T \) is graded-mated.
  It is not necessary to assume that \( R \subseteq T \) satisfies graded-GD to
  show that \( R \subseteq T \) is mated implies that \( R \subseteq T \) is a
  graded i-extension; this holds for any graded ring extension.
\end{remark}

In the following proposition, we give a graded analog of part of~\cite[Theorem 3.6]{Dob-going-down-simple2} where it is shown that for an integrally closed domain $R$, $R$ is Pr\"{u}fer if and only if it is mated.

\begin{proposition}
  Let \(R\) be an integrally closed graded integral domain.
  Then \(R\) is gr-Pr\"{u}fer if and only if \(R\) is gr-mated.
\end{proposition}

\begin{proof}
  Let \(R\) be an integrally closed graded integral domain.
  Suppose that \(R\) is a gr-Pr\"{u}fer domain.
  Let \(P\) be a graded prime ideal of \(R\) and \(T\) a graded overring of \(R\).
  By~\cite[Remark 2.9]{Kla-grperi}, \(T\) satisfies gr-GD\@.
  By~\cite[Theorem 3.1]{ACZ-grpruf}, $R_P$ is a valuation domain. It follows by~\cite[Theorem 19.15]{Gilmer}, it follows that \(R\subseteq T\) is an \(i\)-extension.
  Thus by Proposition~\ref{gr-pap76-lemma2.2}, it follows that \(R \subseteq T\) is gr-mated.

  Conversely, suppose that \(R\) is gr-mated. Let \(T\) be a graded overring of \(R\) and \(M \in h\text{-Spec}(T)\).
  Then \(P = M \cap R\) is a graded prime ideal of \(R\) such that \(PT \neq T\).
  So there exists at most one graded prime ideal of \(T\) lying over \(P\) because \(R\) is gr-mated.
  By~\cite[Theorem 19.15]{Gilmer}, it follows that \(R_P\) is a valuation domain and hence \(R_{(P)}\) is a gr-valuation
  domain by~\cite[Lemma 4.3]{Sah-chargrpvmd}.
  It follows that \(T_{(M)} = R_{(P)}\) is a gr-valuation domain by~\cite[Theorem 2.3]{AAC-grval}.
  Thus by~\cite[Proposition 2.7]{Kla-grperi}, \(T\) is flat over \(R\) and hence \(R\) is gr-Pr\"{u}fer
  by~\cite[Remark 2.9]{Kla-grperi}.
\end{proof}

\section{Characterization of Graded Injective Domains}

In the following proposition, we prove a graded analog of~\cite[Proposition 2.14]{Pap-topologically}
\begin{proposition}
  Let \(R\) be a \(\Gamma \)-graded domain.
  Then \(R\) is a gr-i-domain if and only if \(R \subseteq R'\) is a gr-\(i\)-extension and \(R^\prime \) is
  gr-Pr\"{u}fer domain.
\end{proposition}

\begin{proof}
  Let \(R\) be a \(\Gamma \)-graded domain.
  The proof is parallel to the proof of~\cite[Proposition 2.14]{Pap-topologically} with appropriate graded substitutions.
  We include the details below for the reader's convenience.

  Suppose that \(R\) is a gr-\(i\)-domain.
  Then \(R \subseteq R^\prime \) is a gr-\(i\)-extension by definition of gr-\(i\)-domain.
  Also, \(R^\prime \) is gr-Pr\"{u}fer by~\cite[Theorem 19.15]{Gilmer} and~\cite[Theorem 3.1]{ACZ-grpruf}.

  Suppose \(R = R^\prime \) and \(R\) is gr-Pr\"{u}fer.
  Then it follows from~\cite[Theorem 3.5(1)]{ACZ-grpruf} that for any graded overring \(S\) of \(R\) there will be at most
  one graded prime ideal of \(S\) lying over any graded prime ideal of \(R\).
  It follows that \(R\) is a gr-\(i\)-domain.

  Now suppose that \(R \subsetneq R^\prime \) is a gr-\(i\)-extension and \(R^\prime \) is gr-Pr\"{u}fer domain and \(R\)
  is not a gr-\(i\)-domain.
  Then there exists a graded overring \(S\) of \(R\) and distinct graded prime ideals \(P\) and \(Q\) of \(S\) such that
  \(P \cap R = Q \cap R\).
  Consider the following diagram of ring extensions
  \[
    \xymatrix{
      R^\prime \: \mono[r] & R^{\prime}S\\
      R \: \mono[u] \mono[r] & S. \mono[u]
    }
  \]
  By~\cite[Lemma 1.6]{Cha-grintpruferlike}, \(R^\prime \) is a graded overring of \(R\).
  It follows that \(R^\prime S\) is a graded overring of \(S\) and is an integral extension because \(R^\prime \) is an
  integral extension of \(R\).
  Then there exist graded prime ideals \(P^\prime \) and \(Q^\prime \) of \(R^\prime S\) such that \(P^\prime \cap S = P\)
  and \(Q^\prime \cap S = Q\).
  So \[P^\prime \cap R = P \cap R = Q \cap R = Q^\prime \cap R.\]
  Because \(R^\prime \) is a gr-Pr\"{u}fer domain, it follows that \(R^\prime S\) is a gr-Pr\"{u}fer domain
  by~\cite[Theorem 3.5(2)]{ACZ-grpruf}.
  It then follows by~\cite[Theorem 3.1]{ACZ-grpruf} and~\cite[Theorem 19.15]{Gilmer} that \(P^\prime = Q^\prime \) which
  contradicts that \(P \neq Q\).
  Thus \(R\) is a gr-\(i\)-domain.
\end{proof}

\section*{Acknowledgments}

 We are very grateful for some helpful comments given by Neil Epstein which helped to improve this paper.

\bibliographystyle{amsalpha}
\bibliography{references}

\providecommand{\bysame}{\leavevmode\hbox to3em{\hrulefill}\thinspace}
\providecommand{\MR}{\relax\ifhmode\unskip\space\fi MR }
\providecommand{\MRhref}[2]{%
  \href{http://www.ams.org/mathscinet-getitem?mr=#1}{#2}
}
\providecommand{\href}[2]{#2}
\begin{thebibliography}{AAC17}

\bibitem[AAC17]{AAC-grval}
Daniel~D. Anderson, David~F. Anderson, and Gyu~Whan Chang,
  \emph{Graded-valuation domains}, Communications in Algebra \textbf{45}
  (2017), no.~9, 4018--4029.

\bibitem[ACZ18]{ACZ-grpruf}
David~F. Anderson, Gyu~Whan Chang, and Muhammad Zafrullah, \emph{Graded
  {P}r\"{u}fer domains}, Communications in Algebra \textbf{46} (2018), no.~2,
  792--809.

\bibitem[Cha17]{Cha-grintpruferlike}
Gyu~Whan Chang, \emph{Graded integral domains and {P}r\"{u}fer-like domains},
  Journal of the Korean Mathematical Society \textbf{54} (2017), no.~6,
  1733–1757.

\bibitem[DD74]{DobDaw-GDpoly}
Jeffrey Dawson and David~E. Dobbs, \emph{On going down in polynomial rings},
  Canadian Journal of Mathematics \textbf{26} (1974), no.~1, 177--184.

\bibitem[Dob74]{Dob-going-down-simple2}
David~E. Dobbs, \emph{On going-down for simple overrings ii}, Communications in
  Algebra \textbf{1} (1974), no.~6, 439--458.

\bibitem[Gil72]{Gilmer}
R.~Gilmer, \emph{Multiplicative ideal theory}, Marcel Decker, Inc., 1972.

\bibitem[Kla]{Kla-grperi}
Hannah Klawa, \emph{Graded perinormality}, To appear in Journal of Algebra and
  Its Applications, \url{https://doi.org/10.1142/S0219498825502718}.

\bibitem[Mer24]{Mer-propgradover}
Adam Merchant, \emph{Some properties of graded overrings}, Journal of Student
  Research at Indiana University East \textbf{6} (2024), no.~2, 93--100.

\bibitem[Pap76]{Pap-topologically}
Ira~J. Papick, \emph{Topologically defined classes of going-down}, Transactions
  of the American Mathematical Society \textbf{219} (1976), 1--37.

\bibitem[Sah14]{Sah-chargrpvmd}
Parviz Sahandi, \emph{Characterizations of graded {P}r\"{u}fer
  $\star$-multiplication domains}, The Korean Journal of Mathematics
  \textbf{22} (2014), no.~1, 181 -- 206.

\bibitem[SS23]{Sah-gradedgd}
Parviz Sahandi and Nematolla Shirmohammadi, \emph{On graded going down
  domains}, ar{X}iv: 2306.16067 [math.AC], 2023.

\end{thebibliography}

\end{document}